\documentclass[10pt,a4paper]{amsart}

\theoremstyle{plain}
\newtheorem{theorem}{Theorem}[section]
\newtheorem{lemma}[theorem]{Lemma}

\theoremstyle{definition}

\theoremstyle{remark}

\newtheorem{notations}[theorem]{Notations}

\newcommand{\ic}{\ensuremath{\mathcal{I}}}
\newcommand{\oc}{\ensuremath{\mathcal{O}}}

\newcommand{\Ps}{\mathbb{P}}

\newcommand{\bC}{\mathbb{C}}

\def\bin #1#2 {\left( \matrix { #1 \cr #2 \cr } \right) }

\begin{document}

\title[Noether-Lefschetz theory  and N\'eron-Severi  group]
{Noether-Lefschetz theory and N\'eron-Severi  group}

\author{Vincenzo Di Gennaro }
\address{Universit\`a di Roma \lq\lq Tor Vergata\rq\rq, Dipartimento di Matematica,
Via della Ricerca Scientifica, 00133 Roma, Italy.}
\email{digennar@axp.mat.uniroma2.it}

\author{Davide Franco }
\address{Universit\`a di Napoli
\lq\lq Federico II\rq\rq, Dipartimento di Matematica e
Applicazioni \lq\lq R. Caccioppoli\rq\rq, P.le Tecchio 80, 80125
Napoli, Italy.} \email{davide.franco@unina.it}

\abstract Let $Z$ be a closed subscheme of a  smooth complex
projective complete intersection variety $Y\subseteq \Ps^N$, with
$dim\,Y=2r+1\geq 3$. We describe the N\'eron-Severi group
$NS_r(X)$ of a general smooth hypersurface $X\subset Y$ of
sufficiently large degree containing $Z$.

\bigskip\noindent {\it{Keywords}}: Noether-Lefschetz Theory, N\'eron-Severi  group,
Milnor fibration, Semistable degenerations.

\medskip\noindent {\it{MSC2000}}\,: 14C25, 14D05, 32S55, 14C20,
14C30, 14M10.

\endabstract
\maketitle

\section{Introduction}

After Noether and Lefschetz \cite{Del} one knows that the
intermediate N\'eron-Severi group (i.e.  the image of the cycle
map $A_r(X)\to H_{m}(X;\mathbb{Z})$, \cite{Fulton}, $\S 19.1$) of
a general hypersurface  $X\subset\Ps^{m+1}=\Ps^{m+1}(\mathbb{C})$,
$m=2r\geq 2 $, is generated by the $r$-th power of the hyperplane
class, as soon as $deg\,X\geq 2+2/r$. On the other hand  a result
of Lopez \cite{Lopez}, inspired by a previous work of Griffiths
and Harris \cite{GH}, gives a recipe for the computation of the
N\'eron-Severi group of a general complex surface of sufficiently
large degree in $\Ps^3$ containing a given smooth curve. One of
the main purposes of our paper is to generalize Lopez's result to
higher N\'eron-Severi groups.

More generally let $Y\subseteq \Ps^N$  be a smooth complex
projective complete intersection of dimension $m+1=2r+1\geq 3$,
$Z$ be a closed subscheme of $Y$, and  $\delta$ be a positive
integer such that $\mathcal I_{Z,Y}(\delta)$ is generated by
global sections. Assume that for $d\gg 0$ the general divisor $X
\in |H^0(Y,\ic_{Z,Y}(d))|$ is smooth. This implies that
$2\,dim\,Z\leq m$ and that, for any $d\geq \delta$, there exists a
smooth hypersurface of degree $d$ containing $Z$ \cite{OS}.
Improving (\cite{OS}, 0.4. Theorem), in \cite{DGF} we proved that
{\it{if $d\geq \delta +1$ then the monodromy representation on
$H^m(X;\mathbb Q)_{\perp Z}^{\text{van}}$ for the family of smooth
divisors $X_t\in |H^0(Y,\mathcal O_Y(d))|$ containing $Z$ as above
is irreducible}} (here we denote by $H^m(X;\mathbb Q)_{\perp
Z}^{\text{van}}$ the orthogonal complement in $H^m(X;\mathbb Q)$
of the subspace $H^m(Y;\mathbb Q)+H^m(X;\mathbb Q)_Z$, where
$H^m(X;\mathbb Q)_Z$ denotes the subspace of $H^m(X;\mathbb Q)$
generated by the cycle classes  of the maximal dimensional
irreducible components of $Z$ if $m=2\,dim\,Z$, and $H^m(X;\mathbb
Q)_Z=0$ otherwise). Such a result does not apply directly to
N\'eron-Severi groups since they involve integral homology, so
this  problem is more subtle.

Unfortunately, the techniques generally used around questions of
this kind (semi-stable degenerations, mixed Hodge structures,
intersection (co)homology, etc.) may have a worse behaviour in
passing from $\mathbb{Q}$ to $\mathbb{Z}$. Our approach consists
in looking at $X$ as the general hyperplane section of a very
special variety with isolated singularities, and to reduce the
problem to a question about the homology groups of such a variety,
which can be expressly computed.

Our first result is the following:
\begin{theorem}
\label{thm1} Let $Y$ and $Z$ be as above, and let $X\in
|H^0(Y,\mathcal O_Y(d))|$ be a general divisor containing $Z$,
with $d\geq \delta +1$. Let $Z_1,\dots,Z_{\rho}$ be the
irreducible components of $Z$ of dimension $r$ (if there are).
Assume that  $h^{i,m-i}(X)\not= 0$ for some $i\not= r$. Then the
N\'eron-Severi group $NS_r(X)$ of $X$ is of rank $\rho +1$ freely
generated by $Z_1,\dots , Z_{\rho}$ and by the class $H_X^r\in
H^m(X;\mathbb Z)$ of the linear section of $X$.
\end{theorem}
\noindent Theorem above follows from the more general Theorem
\ref{thmb} which improves the main  result of \cite{DGF}, and also
provides us the appropriate context in order to extend Lopez
Theorem to higher N\'eron-Severi groups:
\begin{theorem}
\label{thm2} Let $Y$ be as above. Let $X,G_1,\dots,G_r$ be a
regular sequence of smooth divisors in $Y$, with $X \in
|H^0(Y,\mathcal O_Y(d))|$, $G_l\in |H^0(Y,\mathcal O_Y(k_l))|$ for
$1\leq l\leq r$, and $d>k_1>\dots>k_r\geq 1$. Let $Z$ be a closed
subscheme of the complete intersection $X\cap G_1\cap\dots \cap
G_r$. Let $X_t\in |H^0(Y,\mathcal O_Y(d))|$ be a general divisor
containing $Z$, and assume that $h^{i,m-i}(X)\not= 0$ for some
$i\not= r$. Denote by $Z_1,\dots,Z_{\rho},R_1,\dots,R_{\sigma}$
the irreducible components of $X_t\cap G_1\cap\dots \cap G_r$,
where $Z_1,\dots,Z_{\rho}$ denote the irreducible components of
$Z$ of maximal dimension $r$. Then  $NS_r(X_t)$  is of rank $\rho
+ \sigma$ freely generated by
$H^r_{X_t},Z_1,\dots,Z_{\rho},R_1,\dots,R_{\sigma -1 }$.
\end{theorem}
\noindent When $Y=\Ps^3$, and $Z\subseteq Y$ is an irreducible
smooth curve, then previous Theorem \ref{thm2} reduces to the main
result in \cite{Lopez}, of which we now give a different proof.

We note that the proof of the asserted independence of the
generators will follow from the injectivity (proved in Theorem
\ref{thma}) of the push-forward $ H_m(W;\mathbb Z)\to
H_m(X;\mathbb Z)$, where $W=X\cap G$ is a complete intersection
with $G$ smooth hypersurface such that $deg\,G<deg\,X$ (condition
which implies that $W$ has at worst isolated singularities
\cite{Vogel}). We believe this fact of independent interest since
it can be false when $G$ is not smooth, even for nodal $W$'s.
Indeed a general hypersurface  $X\subset \Ps^5$ contains
non-factorial nodal threefolds of the form $G\cap X$, with
$deg\,G<deg\,X$.

\section{Preliminaries}

\begin{notations}\label{notazioni}
(i) Let $Y\subseteq \Ps^N$  be a smooth complex projective
variety, complete intersection in $\Ps^N$, of odd dimension
$m+1=2r+1\geq 3$. Fix integers $1\leq k<d$, and smooth divisors
$G\in |H^0(Y,\mathcal O_Y(k))|$ and $X\in |H^0(Y,\mathcal
O_Y(d))|$. $X$ is a smooth complete intersection, hence
$H^m(X;\mathbb Z)$ is torsion free and  $H^m(X;\mathbb Z)\subset
H^m(X;\mathbb Q)$. Put $W:=G\cap X$. By (\cite{Vogel}, p. 133,
Proposition 4.2.6. and proof) we know that $W$ has only isolated
singularities.

(ii) For any smooth divisor $X_t\in |H^0(Y,\mathcal O_Y(d))|$ we
will denote by $PD$  the Poincar\'e dualities $H_m(X_t;\mathbb
Z)\cong H^m(X_t;\mathbb Z)$ and $H_m(X_t;\mathbb Q)\cong
H^m(X_t;\mathbb Q)$.

(iii) Let $P=Bl_W(Y)$ be the blowing-up of $Y$ along $W$. We refer
to \cite{DGF}, 4.1, for a geometric description of $P$. Here we
recall some facts. $P$ has  at worst isolated singularities, which
are isolated hypersurface singularities because $P$ is a divisor
in the smooth variety $\Ps(\oc_{Y}(k)\oplus\oc_{Y}(d))$. For the
strict transform $\tilde G$ of $G$ in $P$ we have $\tilde G\cong
G$, and $G\subseteq P\backslash Sing(P)$. The same holds for any
smooth divisor $X_t\in |H^0(Y,\mathcal O_Y(d))|$ containing $W$.

(iv) Let $\{X_t\}_{t\in L}$ ($L=\Ps^1$) be a general pencil of
divisors $X_t\in |H^0(Y,\mathcal O_Y(d))|$ containing $W$, and let
$B_L$ be its base locus (apart from $W$), which we may consider as
contained in $P\backslash Sing(P)$. Notice that $B_L\cong X_t\cap
M_L$ for a suitable general $M_L\in|H^0(Y,\oc_{Y}(d-k))|$. Denote
by $P_L$ the blowing-up of $P$ along $B_L$, and consider the
natural map $f:P_L\to L$. Let $A\subseteq L$ be the set of the
critical values of $f$. This is a finite set, and there is a point
$a_{\infty}\in A$ corresponding to the unique reducible fibre of
$f$. For any $t\in L\backslash \{a_{\infty}\}$ we have
$f^{-1}(t)\cong X_t\in |H^0(Y,\mathcal O_Y(d))|$. When $a\in
A\backslash \{a_{\infty}\}$ then $X_a$ is an irreducible divisor
with a unique singular point $q_a\in P_L$, and $Sing(P_L)\subseteq
\{q_a\,:\, a\in A\backslash \{a_{\infty}\}\}$.

(v) Fix two regular values $t\,,t_1\in L\backslash A$, with $t\neq
t_1$, and  for any critical value $a\in A$ of $L$ fix a closed
disk $\Delta_a\subset L\backslash\{t\} \cong \bC$ with center $a$
and radius $0<\rho \ll 1$. As in \cite{La}, (5.3.1) and (5.3.2),
one  proves that
\begin{equation}\label{localiz}
H_{*}(P_L \backslash X_t,X_{t_1} ; \mathbb{Z})\cong{\oplus}_{a\in
A} H_{*} (f^{-1}(\Delta_a),X_{a+\rho} ; \mathbb{Z}).
\end{equation}

(vi)  Fix $a\in A\backslash \{a_{\infty}\}$. By \cite{Loj}, p. 28,
we know that near to the isolated singular point $q_a\in P_L$ the
pencil $f: P_L\to L$ defines a Milnor fibration with Milnor fiber
$X_{a+\rho}\cap D_a$, where $D_a$ denotes a closed ball of
$\Ps(\oc_{Y}(k)\oplus\oc_{Y}(d))$ with center $q_a$ and  small
radius $\epsilon$ with $\rho << \epsilon<< 1$. The Milnor fiber
$X_{a+\rho}\cap D_a$ has the homotopy type of a bouquet of
$m-$spheres. In particular
\begin{equation} \label{dMilnorfiber}
H_{m}(X_{a+\rho}\cap D_a ;\mathbb{Z}) \,\,\,{\text{is torsion
free, and}}\,\,\, H_{m+1}(X_{a+\rho}\cap D_a ;\mathbb{Z})=0.
\end{equation}
Moreover, since $f^{-1}(\Delta_a)\backslash D^{\circ}_a\to
\Delta_a$ is a trivial fibre bundle ($D^{\circ}_a$:= interior of
$D_a$), Excision Axiom and Leray-Hirsch Theorem \cite{Spanier}
imply that the inclusion $(X_{a+\rho}, X_{a+\rho}\cap
D_a)\subseteq (f^{-1}(\Delta_a),f^{-1}(\Delta_a)\cap D_a)$ induces
a natural isomorphism $H_{*}(X_{a+\rho}, X_{a+\rho}\cap D_a
;\mathbb{Z})\cong H_{*}(f^{-1}(\Delta_a),f^{-1}(\Delta_a)\cap D_a
;\mathbb{Z})$. Therefore, from the conic structure of
$f^{-1}(\Delta_a)\cap D_a$ (\cite{Loj}, Lemma (2.10)), we deduce a
natural isomorphism for any $l>0$
\begin{equation}
\label{tMilnorfiber} H_{l}(X_{a+\rho}, X_{a+\rho}\cap D_a
;\mathbb{Z})\cong H_{l}(f^{-1}(\Delta_a);\mathbb{Z}).
\end{equation}

(vii) The fibre $f^{-1}(a_{\infty})$ is the union $\widetilde{G}
\cup \widetilde{M_L}$ of the strict transforms of $G$ and $M_L$ in
$P_L$. Since  $\widetilde{G}\cong G$ and $\widetilde{M_L}$ is
isomorphic to the blowing-up of $M_L$  along the smooth complete
intersection $M_L\cap W$, and $\widetilde{G} \cap
\widetilde{M_L}\cong G\cap M_L$, then the map
$f^{-1}(\Delta_{a_{\infty}})\to \Delta_{a_{\infty}}$ is a
semi-stable degeneration. Hence $f^{-1}(\Delta_{a_{\infty}})$
retracts onto $f^{-1}({a_{\infty}})$ (\cite{PS}, p. 185), and  we
have
\begin{equation}
\label{retract} H_{*}(f^{-1}(\Delta_{a_{\infty}})
;\mathbb{Z})\cong H_{*}(f^{-1}({a_{\infty}});\mathbb{Z}).
\end{equation}

(viii) For any regular value $t\in L\backslash A$ the natural
inclusion map $i_t:X_t\to P_L$ induces Gysin maps $i_t^{\star}:
H_{m+2}(P_L;\mathbb Z)\to H_{m}(X_t;\mathbb Z)$ and $i_{t,\mathbb
Q}^{\star}: H_{m+2}(P_L;\mathbb Q)\to H_{m}(X_t;\mathbb Q)$. Now
let $I_W$ be  the subspace of the invariant cocycles in
$H^m(X_t;\mathbb Q)$ with respect to the monodromy representation
on $H^m(X_t;\mathbb Q)$ for the family of smooth divisors $X_t\in
|H^0(Y,\mathcal O_Y(d))|$ containing $W$. Denote by $i_W:W\to X_t$
the inclusion map. It induces push-forward maps ${i_{W}}_*:
H_m(W;\mathbb Z)\to H_m(X_t;\mathbb Z)$ and ${i_{W,\mathbb{Q}}}_*:
H_m(W;\mathbb Q)\to H_m(X_t;\mathbb Q)$. Put $H^m(X_t;\mathbb
Z)_W:= Im(PD\circ {i_{W}}_*)$, and $H^m(X_t;\mathbb Q)_W:= Im(PD
\circ {i_{W,\mathbb{Q}}}_*)$. Observe that $H^m(Y;\mathbb Z)+
H^m(X_t;\mathbb Z)_W$ is contained in $Im(PD\circ i_t^{\star})$.
By \cite{DGF}, (11), we also know that
\begin{equation}\label{inv}
I_W=H^m(Y;\mathbb Q)+ H^m(X_t;\mathbb Q)_W.
\end{equation}
\end{notations}

\begin{lemma}\label{coker} The following properties hold:
(a) $Im(PD\circ i_t^{\star})=H^m(Y;\mathbb Z)+ H^m(X_t;\mathbb
Z)_W$; (b) $Coker(PD\circ i_t^{\star})$ is torsion free; (c)
$dim_{\mathbb Q}\,Ker(PD\circ i_{t,\mathbb Q}^{\star})\leq 2$.
\end{lemma}
\begin{proof}

(a) We only have to prove that
\begin{equation}\label{tincl}
Im(PD\circ i_t^{\star})\subseteq H^m(Y;\mathbb Z)+ H^m(X_t;\mathbb
Z)_W.
\end{equation}
Denote by $\widetilde W$ and $\widetilde{B_L}$ the inverse images
of $W\subseteq Y$ and $B_L\subseteq Bl_W(Y)$ in $P_L$. The map
$P_L\to Y$ induces an isomorphism
$\alpha_1:P_L\backslash(\widetilde W\cup \widetilde{B_L})\to
Y\backslash (W\cup(X_t\cap M_L))$. Consider the following natural
commutative diagram:
$$
\begin{array}{ccc}
H_{m+2}(P_L;\mathbb Z)&\stackrel{\rho_1}{\to}& H_{m+2}(P_L\backslash(\widetilde W\cup \widetilde{B_L});\mathbb Z) \\
\stackrel \alpha{}\downarrow &     &  \Vert\stackrel {\alpha_1}{}\\
H_{m+2}(Y;\mathbb Z)&\stackrel{\rho_2}{\to}& H_{m+2}(Y\backslash (W\cup(X_t\cap M_L));\mathbb Z) \\
\stackrel \beta{}\downarrow &     &  \downarrow\stackrel {\beta_1}{} \\
H_m(X_t;\mathbb Z)&\stackrel{\rho_3}{\to}& H_m(X_t\backslash (W\cup(X_t\cap M_L));\mathbb Z) \\
\end{array}
$$
where $\alpha$ and $\alpha_1$ are push-forward maps, $\beta$ and
$\beta_1$ are Gysin maps, and $\rho_1$, $\rho_2$ and $\rho_3$ are
restriction maps to open subset in Borel-Moore homology
(\cite{Fulton}, p. 371). Fix $c\in Im(PD\circ i_t^{\star})$, and
let $c'\in H_{m+2}(P_L;\mathbb Z)$ such that $(PD\circ
i_t^{\star})(c')=c$. Since
$\beta_1\circ\alpha_1\circ\rho_1=\rho_3\circ i_t^{\star}$ then we
have $\rho_3((PD)^{-1}(c))= (\rho_3\circ\beta\circ\alpha)(c')$.
Hence we have $(PD)^{-1}(c)-\beta(\alpha(c'))\in Ker(\rho_3)=
Im(H_m(W\cup(X_t\cap M_L);\mathbb Z)\to H_m(X_t;\mathbb Z))$
(\cite{Fulton}, p. 371, (6)). So to prove  (\ref{tincl}) it
suffices to prove that $Im(H_m(W\cup(X_t\cap M_L);\mathbb Z)\to
H_m(X_t;\mathbb Z)\cong H^m(X_t;\mathbb Z))$ is contained in
$H^m(Y;\mathbb Z)+H^m(X_t;\mathbb Z)_W$. Since $W$ has only
isolated singularities, and  $M_L$ is general, then $W\cap M_L$
and $X_t\cap M_L$  are smooth complete intersections in $\Ps^N$,
of dimension $m-2$ and $m-1$. In particular $H_{m-1}(W \cap
M_L;\mathbb Z)=0$. From the Mayer-Vietoris sequence of the pair
$(W, X_t\cap M_L)$ we deduce that the natural map $H_{m}(W;\mathbb
Z)\oplus H_{m}(X_t\cap M_L;\mathbb Z) \to H_{m}(W\cup(X_t\cap
M_L);\mathbb Z)$ is surjective. Therefore to prove (\ref{tincl})
it suffices to prove that $Im(H_m(X_t\cap M_L;\mathbb Z)\to
H_m(X_t;\mathbb Z)\cong H^m(X_t;\mathbb Z))$ is contained in
$H^m(Y;\mathbb Z)$. This is obvious because $H_m(X_t\cap
M_L;\mathbb Z)$ is generated by the linear section class.

\smallskip
(b) Since $X_t\subseteq P_L\backslash Sing(P_L)$, using Excision
Axiom and Duality  Theorem (\cite{Spanier}, p. 296) we get a
natural isomorphism $H_{m+2}(P_L,P_L\backslash X_t;
\mathbb{Z})\cong H^{m}(X_t; \mathbb{Z})$. Therefore $PD\circ
i_t^{\star}$ identifies with the second map of the natural exact
sequence
\begin{equation}\label{seq}
H_{l}(P_L\backslash X_t; \mathbb{Z})\to H_{l}(P_L; \mathbb{Z})\to
H_{l}(P_L,P_L\backslash X_t; \mathbb{Z})\to H_{l-1}(P_L\backslash
X_t; \mathbb{Z})
\end{equation}
when $l=m+2$, and the proof of (b) amounts to show that
$H_{m+1}(P_L\backslash X_t; \mathbb{Z})$ is torsion free. Since
$H_{m+1}(X_{t_1}; \mathbb{Z})=0$, from the natural exact sequence
\begin{equation}\label{dseq}
H_{l}(X_{t_1}; \mathbb{Z})\to H_{l}(P_L\backslash X_t;
\mathbb{Z})\to H_{l}(P_L\backslash X_t, X_{t_1}; \mathbb{Z})
\end{equation}
we see that $H_{m+1}(P_L\backslash X_t; \mathbb{Z})$ is contained
in $H_{m+1}(P_L\backslash X_t, X_{t_1}; \mathbb{Z})$. Hence, by
(\ref{localiz}),  to prove (b) it is enough to show that $H_{m+1}
(f^{-1}(\Delta_a),X_{a+\rho} ; \mathbb{Z})$ is torsion free for
any $a\in A$. To this aim, consider the exact sequence of the pair
$(f^{-1}(\Delta_a),X_{a+\rho})$
\begin{equation}\label{teq}
H_{l} (X_{a+\rho}; \mathbb{Z})\to H_{l} (f^{-1}(\Delta_a);
\mathbb{Z}) \to H_{l} (f^{-1}(\Delta_a),X_{a+\rho}; \mathbb{Z})\to
H_{l-1} (X_{a+\rho}; \mathbb{Z}).
\end{equation}
Since $H_{m+1} (X_{a+\rho}; \mathbb{Z})=0$ and $H_{m} (X_{a+\rho};
\mathbb{Z})$ is torsion free then it suffices to prove  that
\begin{equation}\label{suff}
H_{m+1} (f^{-1}(\Delta_a); \mathbb{Z})\,\, {\text{is torsion free
for any $a\in A$}}.
\end{equation}

When $a\neq a_{\infty}$ this follows by (\ref{dMilnorfiber}) and
(\ref{tMilnorfiber}), taking into account the natural exact
sequence of the pair $(X_{a+\rho},X_{a+\rho}\cap D_a)$
\begin{equation}\label{qeq}
H_{l} (X_{a+\rho}; \mathbb{Z})\to H_{l}(X_{a+\rho},X_{a+\rho}\cap
D_a ;\mathbb{Z})\to H_{l-1} (X_{a+\rho}\cap D_a; \mathbb{Z}),
\end{equation}
and that $H_{m+1} (X_{a+\rho}; \mathbb{Z})=0$.

When $a=a_{\infty}$ by (\ref{retract}) we see that to prove
(\ref{suff}) it is enough to prove that $H_{m+1}
(f^{-1}(a_{\infty}); \mathbb{Z})$ is torsion free. This follows
from the Mayer-Vietoris sequence
\begin{equation}\label{Mayer}
H_{l} (G; \mathbb{Z})\oplus H_{l} (\widetilde{M_L}; \mathbb{Z})\to
H_{l} (f^{-1}(a_{\infty}); \mathbb{Z})\to H_{l-1} (G\cap M_L;
\mathbb{Z})
\end{equation}
because $H_{m} (G\cap M_L; \mathbb{Z})$ is torsion free, $H_{m+1}
(G; \mathbb{Z})=0$, and  $H_{m+1} (\widetilde{M_L}; \mathbb{Z})$
$=0$ in view of the decomposition (\cite{Voisin}, p. 170,
Th\'eor\`eme 7.31)
\begin{equation}\label{Blowdec}
H_{l} (\widetilde{M_L}; \mathbb{Z})\cong H_{l} ({M_L};
\mathbb{Z})\oplus H_{l-2} ({M_L\cap W}; \mathbb{Z}).
\end{equation}

\smallskip
(c) By the sequence (\ref{seq}) (tensored with $\mathbb Q$) we see
that to prove (c) it is enough to show that $h_{m+2}(P_L\backslash
X_t; \mathbb{Q})\leq 2$. Hence, since $H_{m+2} (X_{t_1} ;
\mathbb{Z})\cong \mathbb Z$, by  (\ref{dseq}) it suffices to prove
that $h_{m+2}(P_L\backslash X_t, X_{t_1}; \mathbb{Q})\leq 1$.
Combining (\ref{dMilnorfiber}) and (\ref{tMilnorfiber}) with the
exact sequences (\ref{teq}) and (\ref{qeq}) with $l=m+2$, and with
the fact that $H_{m+1} (X_{a+\rho} ; \mathbb{Z})=0$, we get
$H_{m+2} (f^{-1}(\Delta_a),X_{a+\rho} ; \mathbb{Q})=0$ for any
$a\in A\backslash\{a_{\infty}\}$. So, in view of the decomposition
(\ref{localiz}), the proof of (c) amounts to prove that $h_{m+2}
(f^{-1}(\Delta_{a_{\infty}}),X_{a_{\infty}+\rho} ; \mathbb{Q})\leq
1$. To this aim observe that since $H_{m+2} ( X_{a_{\infty}+\rho}
; \mathbb{Q})$ is generated by the linear section class then the
first map in (\ref{teq}) with $l=m+2$  is injective. It follows
that $h_{m+2} (f^{-1}(\Delta_{a_{\infty}}),X_{a_{\infty}+\rho} ;
\mathbb{Q})=h_{m+2} (f^{-1}(\Delta_{a_{\infty}}); \mathbb{Q})-1$
for $H_{m+1} (X_{a_{\infty}+\rho} ; \mathbb{Z})=0$. Therefore, by
(\ref{retract}), we see that it is enough to prove that $h_{m+2}
(f^{-1}({a_{\infty}}); \mathbb{Q})=2$. This follows by the
sequence (\ref{Mayer}) with $l=m+2$, taking into account that the
kernel of its first map is $H_{m+2} (G\cap {M_L}; \mathbb{Z})\cong
\mathbb Z$, that $H_{m+1} (G\cap {M_L}; \mathbb{Z})=0$, and that $
h_{m+2} (G; \mathbb{Q})+h_{m+2} (\widetilde{M_L}; \mathbb{Q})=3$
in view of (\ref{Blowdec}).
\end{proof}

\begin{theorem}
\label{thma} Let $Y\subseteq \Ps^N$ be  a smooth complete
intersection of odd dimension $m+1=2r+1\geq 3$. Fix integers
$1\leq k<d$, and smooth divisors $G\in |H^0(Y,\mathcal O_Y(k))|$
and $X\in |H^0(Y,\mathcal O_Y(d))|$. Put $W:= X\cap G$. Then the
following properties hold:

(a) $I_W\cap H^m(X;\mathbb Z)=H^m(Y;\mathbb Z)+H^m(X;\mathbb Z)_W$
(see  Notations \ref{notazioni}, (viii));

(b) the push-forward ${i_{W}}_*: H_m(W;\mathbb Z)\to H_m(X;\mathbb
Z)$ is injective.
\end{theorem}

\begin{proof} (a) By (\ref{inv}) we only have  to prove that $I_W\cap H^m(X;\mathbb
Z)\subseteq H^m(Y;\mathbb Z)+H^m(X;\mathbb Z)_W$. To this aim,
consider $\xi\in I_W\cap H^m(X;\mathbb Z)$. By (\ref{inv})  $\xi$
is a torsion element in $H^m(X;\mathbb Z)$ modulo $H^m(Y;\mathbb
Z)+H^m(X;\mathbb Z)_W$. So, by Lemma \ref{coker}, (a) and (b), we
deduce that $\xi \in H^m(Y;\mathbb Z)+H^m(X;\mathbb Z)_W$.

(b) Since $H_m(W;\mathbb Z)$ is torsion free (\cite{Dimca}, (4.4)
Corollary (i), p. 162)) then to prove (b) is equivalent to prove
that ${i_{W,\mathbb{Q}}}_*$ is injective. By (\ref{inv}) we have
$I_W=H^m(X;\mathbb Q)_W$ because $H^m(Y;\mathbb Q)$ is generated
by the linear section class which belongs also to $H^m(X;\mathbb
Q)_W$. Therefore to show (b) it suffices to show that
$h_m(W;\mathbb Q)\leq dim_{\mathbb Q }\,I_W$. On the other hand,
by Lemma \ref{coker}, (a), and (\ref{inv}), we have
$I_W=Im(PD\circ i_{t,\mathbb Q}^{\star})$. Hence $dim_{\mathbb Q
}\,I_W=h_{m+2}(P_L;\mathbb Q) - dim_{\mathbb Q }\,Ker(PD\circ
i_{t,\mathbb Q}^{\star})$, and, in view of Lemma \ref{coker}, (c),
in order to prove (b) it is enough to show that
\begin{equation}\label{csuff}
h_{m+2}(P_L;\mathbb Q)\geq h_m(W;\mathbb Q)+2.
\end{equation}
To this purpose let $R\to P$ be a desingularization of $P$, and
$R_L\to P_L$ be the induced map on the blowing-up. By
(\cite{Dimca2}, Proposition 5.4.4 p. 157, and Corollary 5.4.11 p.
161) and (\cite{Voisin}, p. 170, Th\'eor\`eme 7.31) we see that
$H^{m+2}(P;\mathbb Q)$ and $H^{m+2}(P_L;\mathbb Q)$ are naturally
embedded in $H^{m+2}(R_L;\mathbb Q)$ via pull-back. Therefore the
pull-back $H^{m+2}(P;\mathbb Q)\to H^{m+2}(P_L;\mathbb Q)$ is
injective, hence the push-forward $H_{m+2}(P_L;\mathbb Q)\to
H_{m+2}(P;\mathbb Q)$ is surjective. This map cannot be injective
because the class $\Ps^1\times H_{B_L}^{r-1}\in
H_{m+2}(P_L;\mathbb Q)\backslash\{0\}$ ($H_{B_L}=$ hyperplane
section of $B_L$) vanishes in $H_{m+2}(P;\mathbb Q)$. We deduce
\begin{equation}\label{ucsuff}
h_{m+2}(P_L;\mathbb Q)\geq h_{m+2}(P;\mathbb Q)+1.
\end{equation}
Now let $\widetilde W$ be the exceptional divisor in $P$, and
consider the following exact sequences in Borel-Moore homology
(\cite{Fulton}, p. 371, (6)):
$$
0=H_{m+3}(Y;\mathbb Q)\to H_{m+3}(Y\backslash W;\mathbb Q)\to
H_{m+2}(W;\mathbb Q)\stackrel{\nu}\to H_{m+2}(Y;\mathbb Q),
$$
$$
\qquad\qquad\qquad\qquad\,\,\,\,\,\, H_{m+3}(P\backslash
\widetilde W;\mathbb Q)\to H_{m+2}(\widetilde W;\mathbb Q)\to
H_{m+2}(P;\mathbb Q).
$$
By (\cite{Dimca}, p. 161, (4.3) Theorem (i)) we know that
$H_{m+2}(W;\mathbb Q)$ is generated by the linear section class,
so the push-forward $\nu$ is injective, and then the first
sequence proves that $H_{m+3}(Y\backslash W;\mathbb Q)=0$. Then
also $H_{m+3}(P\backslash \widetilde W;\mathbb Q)=0$ because
$P\backslash \widetilde W\cong Y\backslash W$, and from the second
sequence we obtain
\begin{equation}\label{dcsuff}
h_{m+2}(P;\mathbb Q)\geq h_{m+2}(\widetilde W;\mathbb Q).
\end{equation}
Finally, using Leray-Hirsch Theorem \cite{Spanier}, we have
\begin{equation}\label{tcsuff}
h_{m+2}(\widetilde W;\mathbb Q)\geq h_{m}(W;\mathbb Q)+1.
\end{equation}
Putting together (\ref{ucsuff}), (\ref{dcsuff}) and
(\ref{tcsuff}), we get (\ref{csuff}).
\end{proof}

\section{Proof of the announced results}

We keep the same notation we introduced before, and need further
preliminaries.

\begin{notations}\label{duenotazioni}
(i) Let $X,G_1,\dots,G_r$ be a regular sequence of divisors in
$Y$, with  $X \in |H^0(Y,\mathcal O_Y(d))|$ and $G_l\in
|H^0(Y,\mathcal O_Y(k_l))|$ for $1\leq l\leq r$. Let
$\Delta:=X\cap G_1\cap \dots\cap G_r$ be their complete
intersection, hence $dim\,\Delta =r$. Denote by
$C_1,\dots,C_{\omega}$ the irreducible components of $\Delta$.
Assume that $d>k_l$ for any $1\leq l\leq r$, that  $X$ and $G_1$
are smooth,  and that for any $2\leq l\leq r$ one has
\begin{equation}\label{altOS}
dim\,Sing(G_1\cap \dots\cap G_l)\leq l-2 \quad{\text{and}}\quad
dim\,Sing(X\cap G_1\cap \dots\cap G_l)\leq l-1.
\end{equation}
Observe that $\Delta$ verifies condition (0.1) in \cite{OS}. Put
$W:=X\cap G_1$.

(ii) Let $I_{\Delta}$ be  the subspace of the invariant cocycles
in $H^m(X;\mathbb Q)$ with respect to the monodromy representation
on $H^m(X;\mathbb Q)$ for the family of smooth divisors $X\in
|H^0(Y,\mathcal O_Y(d))|$ containing $\Delta$. Let $V_{\Delta}$ be
its orthogonal complement in $H^m(X;\mathbb Q)$. Denote by
$H^m(X;\mathbb Z)_{\Delta}$ the image of $H_m(\Delta;\mathbb Z)$
in $H^m(X;\mathbb Z)$ via the natural map $H_m(W;\mathbb Z)\to
H_m(X;\mathbb Z)\cong H^m(X;\mathbb Z)$.  In Notations
\ref{notazioni}, (viii), we already defined $I_W$ and
$H^m(X;\mathbb Z)_W$. Observe that $I_{\Delta}\subseteq I_W$ for
the monodromy group of the family of smooth divisors $X\in
|H^0(Y,\mathcal O_Y(d))|$ containing $W$ is a subgroup of the
monodromy group of the family of smooth $X\in |H^0(Y,\mathcal
O_Y(d))|$ containing $\Delta$.

(iii) For any $1\leq l\leq r-1$ fix general divisor $H_l\in
|H^0(Y,\mathcal O_Y(\mu_{l}))|$, with $0\ll \mu_1\ll\dots\ll
\mu_{r-1}$, and for any $0\leq l\leq r-1$ define
$(Y_l,X_l,W_l,\Delta_l)$ as follows. For $l=0$ put
$(Y_0,X_0,W_0,\Delta_0):=(Y,X,W,\Delta)$. For $1\leq l\leq r-1$
put $Y_l:=G_1\cap\dots\cap G_l\cap H_1\cap\dots\cap H_l$,
$X_l:=X\cap Y_l$, $W_l:=X\cap Y_l\cap G_{l+1}$, and
$\Delta_l:=\Delta\cap Y_l$. Notice that  $dim\,Y_{r-1}=3$ and that
$\Delta_{r-1}=W_{r-1}$.
\end{notations}

\begin{theorem}\label{thmb}
Let $X\in |H^0(Y,\mathcal O_Y(d))|$ be a general divisor
containing $\Delta$. Then $H^m(Y;\mathbb Z)+H^m(X;\mathbb
Z)_{\Delta}$ is freely generated by $C_1,\dots,C_{\omega -1}$ and
the linear section $H^r_X$. Moreover $I_{\Delta}\cap H^m(X;\mathbb
Z)=H^m(Y;\mathbb Z)+H^m(X;\mathbb Z)_{\Delta}=H^m(Y;\mathbb
Z)+H^m(X;\mathbb Z)_{W}$, and the monodromy representation on
$V_{\Delta}$ for the family of smooth divisors in $|H^0(Y,\mathcal
O_Y(d))|$ containing $\Delta$ is irreducible.
\end{theorem}

\begin{proof} Since $H_m(\Delta;\mathbb Z)$ is freely generated by
$C_1,\dots,C_{\omega}$ then to prove the first assertion  is
equivalent to prove that the push-forward $H_{m}(\Delta;\mathbb
Z)\to H_{m}(X;\mathbb Z)$ is injective. When $r=1$ this follows by
Theorem \ref{thma}, (b), because in this case $W=\Delta$. Now
argue by induction on $r\geq 2$. Since $\Delta_1=\Delta \cap H_1$
then $\Delta$ and $\Delta_{1}$ have the same number of components.
Therefore the Gysin map $H_{m}(\Delta;\mathbb Z)\to
H_{m-2}(\Delta_1;\mathbb Z)$ is bijective, and its composition
$\varphi$ with the push-forward $H_{m-2}(\Delta_1;\mathbb Z)\to
H_{m-2}(X_1;\mathbb Z)$ is injective by induction. On the other
hand $\varphi$ is nothing but the composition of the push-forward
$H_{m}(\Delta;\mathbb Z)\to H_{m}(W;\mathbb Z)$ with the Gysin map
$H_{m}(W;\mathbb Z)\to H_{m-2}(X_1;\mathbb Z)$ (observe that
$X_1=W\cap H_1$). Hence the map $H_{m}(\Delta;\mathbb Z)\to
H_{m}(W;\mathbb Z)$ is injective, and so is the map
$H_{m}(\Delta;\mathbb Z)\to H_{m}(X;\mathbb Z)$ by Theorem
\ref{thma}, (b), again.

As for the remaining claims, note that since $I_{W}\supseteq
I_{\Delta}$ and $I_{\Delta}\cap H^m(X;\mathbb Z)\supseteq
H^m(Y;\mathbb Z)+H^m(X;\mathbb Z)_{\Delta}$, by Theorem
\ref{thma}, (a), it suffices to prove that $H^m(X;\mathbb Z)_{W}$
$\subseteq H^m(Y;\mathbb Z)+H^m(X;\mathbb Z)_{\Delta}$, and that
$V_{\Delta}$ is irreducible. So it is enough to show that for any
$0\leq l\leq r-1$ one has
\begin{equation}\label{ninduzione}
H^{m_l}(X_l;\mathbb Z)_{W_l}\subseteq H^{m_l}(Y_l;\mathbb
Z)+H^{m_l}(X_l;\mathbb Z)_{{\Delta}_l}
\end{equation}
($m_l:=m-2l$), and that the monodromy representation on
$V_{\Delta_{l}}$ for the family of smooth divisors $X_l\in
|H^0(Y_l,\mathcal O_{Y_l}(d))|$ containing $\Delta_l$ is
irreducible. To this purpose we argue by decreasing induction on
$l$. When $l=r-1$ we have $\Delta_{r-1}=W_{r-1}$. In this case
(\ref{ninduzione}) is obvious, and by (\cite{DGF}, Theorem 1.1) we
know that $V_{\Delta_{r-1}}$ is irreducible. Now assume $0\leq
l<r-1$. By induction we have
\begin{equation}\label{nninduzione}
I_{{\Delta}_{l+1}}\cap H^{m_{l+1}}(X_{l+1};\mathbb
Z)=H^{m_{l+1}}(Y_{l+1};\mathbb Z)+H^{m_{l+1}}(X_{l+1};\mathbb
Z)_{{\Delta}_{l+1}}.
\end{equation}
Since $X_{l+1}=W_l\cap H_{l+1}$ then the inclusion map
$i_{X_{l+1}}:X_{l+1}\to W_l$ defines a Gysin map
$i_{X_{l+1}}^{\star}:H_{m_l}(W_l;\mathbb Z)\to
H_{m_{l+1}}(X_{l+1};\mathbb Z)$.  Using the same argument as in
the proof of (\cite{DGF}, Lemma 2.3) we see that $PD\circ
i_{X_{l+1}}^{\star}$ is injective (recall that
$H_{m_l}(W_l;\mathbb Z)$ is torsion free by (\cite{Dimca}, (4.4)
Corollary (i), p. 162)), and that its image is globally invariant.
Since $\Delta_{l}\subseteq W_l$,  $\Delta_{l+1}=\Delta_{l}\cap
H_{l+1}$ (hence the Gysin map $H_{m_l}(\Delta_{l};\mathbb Z)\to
H_{m_{l+1}}(\Delta_{l+1};\mathbb Z)$ is bijective because both
groups are freely generated by the irreducible components), and by
Lefschetz Hyperplane Theorem we have $H^{m_{l+1}}(Y_{l};\mathbb
Z)\cong H^{m_{l+1}}(Y_{l+1};\mathbb Z)$, then $Im\,(PD\circ
i_{X_{l+1}}^{\star})\supseteq H^{m_{l+1}}(Y_{l+1};\mathbb
Z)+H^{m_{l+1}}(X_{l+1};\mathbb Z)_{{\Delta}_{l+1}}$. It follows
that these groups are equal, i.e.
\begin{equation}\label{xinduzione}
H_{m_l}(W_l;\mathbb Z)\cong Im\,(PD\circ i_{X_{l+1}}^{\star})=
H^{m_{l+1}}(Y_{l+1};\mathbb Z)+H^{m_{l+1}}(X_{l+1};\mathbb
Z)_{{\Delta}_{l+1}}.
\end{equation}
In fact, otherwise, by (\ref{nninduzione}) one would have
$V_{\Delta_{l+1}}\cap Im\,(PD\circ i_{X_{l+1},\mathbb
Q}^{\star})\neq \{0\}$, and since $V_{\Delta_{l+1}}$ is
irreducible it would follow that $H_{m_l}(W_l;\mathbb Q)\cong
H^{m_{l+1}}(X_{l+1};\mathbb Q)$. This is impossible because, for
$0\ll \mu_1\ll\dots\ll \mu_{l+1}$, $h_{m_{l+1}}(X_{l+1};\mathbb
Q)$ is arbitrarily large with respect to $h_{m_l}(W_l;\mathbb Q)$.
From (\ref{xinduzione}) we get (\ref{ninduzione}) for the natural
map $H_{m_l}(W_{l};\mathbb Z)\to H_{m_l}(X_{l};\mathbb Z)\cong
H^{m_l}(X_{l};\mathbb Z)$ sends $H^{m_{l+1}}(Y_{l+1};\mathbb Z)$
in $H^{m_{l}}(Y_{l};\mathbb Z)$ and $H^{m_{l+1}}(X_{l+1};\mathbb
Z)_{{\Delta}_{l+1}}$ in $H^{m_{l}}(X_{l};\mathbb
Z)_{{\Delta}_{l}}$. Finally note that by Theorem \ref{thma}  and
(\ref{ninduzione}) it follows $I_{\Delta_l}=I_{W_l}$. So
(\cite{DGF}, Theorem 1.1, and (11)) implies  $V_{\Delta_{l}}$ is
irreducible.
\end{proof}

\smallskip

We are in position to prove the results announced in the
Introduction.

\begin{proof}[Proof of Theorem \ref{thm1}]
Let $G_1,\dots, G_r$ be general divisors in $|H^0(Y,\mathcal
O_Y(\delta))|$ containing $Z$.  By (\cite{OS}, 1.2. Theorem) we
know that  $X,G_1,\dots, G_r$ is a regular sequence, verifying
conditions (\ref{altOS}). Put $\Delta:=X\cap G_1\cap\dots\cap
G_r$. Hence $\Delta$ is a complete intersection  of dimension $r$
containing $Z$,  and again by (\cite{OS}, 1.2. Theorem) we know
that $\Delta\backslash Z$ is smooth and connected. Observe that
$\Delta\neq Z_1\cup\dots\cup Z_{\rho}$, otherwise $Z=\Delta$ and
this is in contrast with the assumption that $\mathcal
I_{Z,Y}(\delta)$ is generated by global sections. Therefore, apart
from $Z_1,\dots ,Z_{\rho}$, $\Delta$ has a unique residual
irreducible component, and so $H^m(Y;\mathbb Z)+H^m(X;\mathbb
Z)_{\Delta}$ is freely generated by $Z_1,\dots ,Z_{\rho},H^r_X$ by
Theorem \ref{thmb}. Again by Theorem \ref{thmb} we deduce that
$I_{\Delta}$ is equal to the subspace $I_Z$ of the invariant
cocycles in $H^m(X;\mathbb Q)$ with respect to the monodromy
representation on $H^m(X;\mathbb Q)$ for the family of smooth
divisors in $|H^0(Y,\mathcal O_Y(d))|$ containing $Z$. Since
${I_Z}^{\perp}$ ($=V_{\Delta}$) is irreducible,  a standard
argument shows that $NS(X)\subseteq I_{\Delta}$, i.e.
$NS(X)=H^m(Y;\mathbb Z)+H^m(X;\mathbb Z)_{\Delta}$.
\end{proof}

\begin{proof}[Proof of Theorem \ref{thm2}]
Put $\Delta := X_t\cap G_1\cap\dots\cap G_r$. Since
$X_t,G_1,\dots,G_r$ are smooth and $d>k_1>\dots >k_r$ then the
conditions (\ref{altOS}) are verified (\cite{Vogel}, p. 133,
Proposition 4.2.6. and proof). Therefore Theorem \ref{thmb}
applies to $\Delta$. We deduce that $NS(X_t)=H^m(Y;\mathbb
Z)+H^m(X_t;\mathbb Z)_{\Delta}$, hence $NS(X_t)$ is freely
generated by $H^r_{X_t},Z_1,\dots,Z_{\rho},R_1,\dots,R_{\sigma -1
}$.
\end{proof}

\bigskip {\bf{Aknowledgements}}

We would like to thank Ciro Ciliberto for valuable discussions and
suggestions on the subject of this paper.

\end{document}